\documentclass{amsart}
\usepackage{amssymb,amsmath,latexsym}
\usepackage{amsthm}
\usepackage{fontenc}
\usepackage{amssymb}
\numberwithin{equation}{section}
\newtheorem{theorem}{Theorem}[section]
\newtheorem{corollary}{Corollary}[theorem]
\newtheorem{lemma}[theorem]{Lemma}

\begin{document}
\author{Alexander E. Patkowski}
\title{On certain Fourier expansions for the Riemann zeta function}

\maketitle
\begin{abstract} We build on a recent paper on Fourier expansions for the Riemann zeta function. We establish Fourier expansions for certain $L$-functions, and offer series representations involving the Whittaker function $W_{\gamma,\mu}(z)$ for the coefficients. Fourier expansions for the reciprocal of the Riemann zeta function are also stated. A new expansion for the Riemann xi function is presented in the third section by constructing an integral formula using Mellin transforms for its Fourier coefficients.
\end{abstract}



\section{Introduction and Main Results} 
The measure 
$$\mu(B):=\frac{1}{2\pi}\int_{B}\frac{dy}{\frac{1}{4}+y^2},$$
for each $B$ in the Borel set $\mathfrak{B},$ has been applied in the work of [7] as well as Coffey [3], providing interesting applications in analytic number theory. For the measure space $(\mathbb{R},\mathfrak{B}, \mu),$ 
\begin{equation} \left\lVert g \right\rVert_2^2:=\int_{\mathbb{R}}|g(t)|^2d\mu,\end{equation}
is the $L^2(\mu)$ norm of $f(x).$ Here (1.1) is finite, and $f(x)$ is measurable [10, pg.326, Definition 11.34]. In a recent paper by Elaissaoui and Guennoun [7], an interesting Fourier expansion was presented which states that, if $f(x)\in L^2(\mu),$ then
\begin{equation} f(x)=\sum_{n\in\mathbb{Z}}a_ne^{-2in\tan^{-1}(2x)},\end{equation}
where
\begin{equation}a_n=\frac{1}{2\pi}\int_{\mathbb{R}}f(y)e^{2in\tan^{-1}(2y)}\frac{dy}{\frac{1}{4}+y^2}.\end{equation}
By selecting $x=\frac{1}{2}\tan(\phi),$ we return to the classical Fourier expansion, since $f(\frac{1}{2}\tan(\phi))$ is periodic in $\pi.$ The main method applied in their paper to compute the constants $a_n$ is the Cauchy residue theorem. However, it is possible (as noted therein) to directly work with the integral
\begin{equation}a_n=\frac{1}{2\pi}\int_{-\pi}^{\pi}f(\frac{1}{2}\tan(\frac{\phi}{2}))e^{in\phi}d\phi.\end{equation}
Many remarkable results were extracted from the Fourier expansion (1.2)--(1.3), including criteria for the Lindel$\ddot{o}$f Hypothesis [7, Theorem 4.6]. \par Let $\rho$ denote the the nontrivial zeros of $\zeta(s)$ in the critical region $(0,1),$ and $\Re(\rho)=\alpha,$ $\Im(\rho)=\beta.$ The goal of this paper is to offer some more applications of (1.2)--(1.3), including a criteria for the Riemann hypothesis. Recall that the Riemann Hypothesis is the statement that $\alpha\notin(\frac{1}{2},1).$
\begin{theorem}\label{thm:thm1} For $\sigma>1,$ $x\in\mathbb{R},$
$$\frac{1}{\zeta(\sigma+ix)}=\frac{1}{\zeta(\sigma+\frac{1}{2})}+\sum_{n\ge1}\bar{a}_ne^{-2in\tan^{-1}(2x)},$$
where $$\bar{a}_n=\frac{1}{n!}\sum_{n> k\ge0}\binom{n}{k}\frac{(-1)^{n}(n-1)!}{(k-1)!}\lim_{s\rightarrow0}\frac{\partial^{k}}{\partial s^k}\frac{1}{\zeta(\sigma+\frac{1}{2}-s)}.$$
 Moreover, if the zeros of $\zeta(s)$ are simple, we have
 $$\frac{1}{\zeta(\sigma-ix)}=\sum_{n\in\mathbb{Z}}\hat{a}_ne^{-2in\tan^{-1}(2x)},$$
 where for $n\ge1,$
$$\hat{a}_n=\frac{1}{n!}\sum_{n\ge k\ge0}\binom{n}{k}\frac{(-1)^{n}(n-1)!}{(k-1)!}\lim_{s\rightarrow0}\frac{\partial^{k}}{\partial s^k}\frac{1}{\zeta(\sigma-\frac{1}{2}+s)}-S(n,\sigma),$$
where
$$S(n,\sigma)=\sum_{\beta: \zeta(\rho)=0}\left(\frac{\sigma-i\beta}{1-\sigma+i\beta}\right)^n\frac{1}{\zeta'(\rho)(1-\sigma+i\beta)(\sigma-i\beta)}$$
$$+\sum_{k\ge1}\left(\frac{\frac{1}{2}+\sigma+2k}{\frac{1}{2}-\sigma-2k}\right)^n\frac{1}{\zeta'(-2k)(\frac{1}{2}-\sigma-2k)(\frac{1}{2}+\sigma+2k)},$$
and $\hat{a}_n=-S(n,\sigma)$ for $n<0,$ $\hat{a}_0=1/\zeta(\sigma+\frac{1}{2}).$

\end{theorem}

\begin{corollary}\label{corollary:Cor1} For $\sigma>1,$
$$\frac{1}{2\pi}\int_{\mathbb{R}}\frac{d\mu}{|\zeta(\sigma+iy)|^2}=\frac{1}{\zeta^2(\sigma+\frac{1}{2})}+\sum_{k\ge1}|\bar{a}_k|^2,$$
where the $\bar{a}_n$ are as defined in the previous theorem. Furthermore, even assuming the Riemann Hypothesis, this integral diverges for $\frac{1}{2}<\sigma<1.$
\end{corollary}

Next we consider a Fourier expansion with coefficients expressed as a series involving the Whittaker function $W_{\gamma, \mu}(z),$ which is a solution to the differential equation [8, pg.1024, eq.(9.220)]
$$\frac{d^2W}{dz^2}+\left(-\frac{1}{4}+\frac{\gamma}{z}+\frac{1-4\mu^2}{4z^2}\right)W=0.$$
This function also has the representation [8, pg.1024, eq.(9.220)]
$$W_{\gamma,\mu}(z)=\frac{\Gamma(-2\mu)}{\Gamma(\frac{1}{2}-\mu-\gamma)}M_{\gamma,\mu}(z)+\frac{\Gamma(2\mu)}{\Gamma(\frac{1}{2}+\mu-\gamma)}M_{\gamma,-\mu}(z).$$
Here the other Whittaker function $M_{\gamma,\mu}(z)$ is given by
$$M_{\gamma,\mu}(z)=z^{\mu+\frac{1}{2}}e^{-z/2}{}_1F_1(\mu-\gamma+\frac{1}{2};2\mu+1;z),$$ where $_1F_1(a;b;z)$ is the well-known confluent hypergeometric function.

\begin{theorem}\label{thm:thm2} Let $v$ be a complex number which may not be an even integer. Then for $1>\sigma>\frac{1}{2},$ we have the expansion

$$\zeta(\sigma+ix)\cos^{v}(\tan^{-1}(2x))=\frac{1}{2}\zeta(\sigma+\frac{1}{2})+\sum_{n\in\mathbb{Z}}\tilde{a}_ne^{-2in\tan^{-1}(2x)},$$
where
$\tilde{a}_n=\frac{(2\sigma^2-4\sigma+\frac{5}{2})}{2(\sigma-\frac{1}{2})^2(\frac{3}{2}-\sigma)^2}\left(\frac{\frac{3}{2}-\sigma}{\sigma-\frac{1}{2}}\right)^n$
for $n<0,$ and for $n\ge1,$
$$\tilde{a}_n=\frac{2\Gamma(v+1)}{\Gamma(\frac{v}{2}+n+1)\Gamma(\frac{v}{2}-n+1)}+\frac{\pi }{2^{v/2+1}}\sum_{k>1} k^{-\sigma}\left(\frac{\log(k)}{2}\right)^{v/2}\frac{W_{n,-\frac{v+1}{2}}(\log(k))}{\Gamma(1+\frac{v}{2}+n)}.$$

\end{theorem}

\section{Proof of Main Theorems}

In our proof of Corollary 1.1.1, we will require a well-known result [13, pg.331, Theorem 11.45] on functions in $L^2(\mathbb{\mu}).$

\begin{lemma} Suppose that $f(x)=\sum_{k\in\mathbb{Z}}a_k\kappa_k,$ where $\{\kappa_n\}$ is a complete orthonormal set and $f(x)\in L^2(\mu),$ 
then 
$$\int_{X}|f(x)|^2d\mu=\sum_{k\in\mathbb{Z}}|a_k|^2.$$

\end{lemma}

\begin{proof}[Proof of Theorem~\ref{thm:thm1}] First we rewrite the integral for $\sigma>1$ as

\begin{equation}\bar{a}_n=\frac{1}{2\pi}\int_{\mathbb{R}}e^{2in\tan^{-1}(2y)}\frac{dy}{\zeta(\sigma+iy)(\frac{1}{4}+y^2)}dy=\frac{1}{2\pi i}\int_{(\frac{1}{2})}\left(\frac{s}{1-s}\right)^n\frac{ds}{\zeta(\sigma-\frac{1}{2}+s)s(1-s)}. \end{equation}
We replace $s$ by $1-s$ and apply the residue theorem by moving the line of integration to the left. By the Leibniz rule, we compute the residue at the pole $s=0$ of order $n+1,$ $n\ge0,$
as \begin{equation}\begin{aligned}
&\frac{1}{n!}\lim_{s\rightarrow0}\frac{d^n}{ds^n}s^{n+1}\left(\left(\frac{1-s}{s}\right)^n\frac{1}{\zeta(\sigma+\frac{1}{2}-s)s(1-s)}\right) \\
&=\frac{1}{n!}\lim_{s\rightarrow0}\frac{d^n}{ds^n}\frac{(1-s)^{n-1}}{\zeta(\sigma+\frac{1}{2}-s)}  \\
&=\frac{1}{n!}\sum_{n\ge k\ge0}\binom{n}{k}\frac{(-1)^n(n-1)!}{(k-1)! }\lim_{s\rightarrow0}\frac{\partial^{k}}{\partial s^k}\frac{1}{\zeta(\sigma+\frac{1}{2}-s)}
 \end{aligned}\end{equation}
The residue at $s=0$ if $n=0$ is $-1/\zeta(\sigma+\frac{1}{2}).$ There are no additional poles when $n<0.$ Since the sum in (2.2) is zero for $k=n$ it reduces to the one stated in the theorem. \par Next we consider the second statement. The integrand in
\begin{equation}\frac{1}{2\pi i}\int_{(\frac{1}{2})}\left(\frac{1-s}{s}\right)^n\frac{ds}{\zeta(\sigma-\frac{1}{2}+s)s(1-s)} \end{equation}
 has simple poles at $s=1-\sigma+i\beta,$ where $\Im(\rho)=\beta.$ The integrand in (2.3) also has simple poles at $s=\frac{1}{2}-\sigma-2k,$ and a pole of order $n+1,$ $n>0,$ at $s=0.$ We compute,
$$\begin{aligned} &\frac{1}{n!}\lim_{s\rightarrow0}\frac{d^n}{ds^n}s^{n+1}\left(\left(\frac{1-s}{s}\right)^n\frac{1}{\zeta(\sigma-\frac{1}{2}+s)s(1-s)}\right) \\
&=\frac{1}{n!}\lim_{s\rightarrow0}\frac{d^n}{ds^n}\frac{(1-s)^{n-1}}{\zeta(\sigma-\frac{1}{2}+s)}  \\
&=\frac{1}{n!}\sum_{n\ge k\ge0}\binom{n}{k}\frac{(-1)^n(n-1)!}{(k-1)!}\lim_{s\rightarrow0}\frac{\partial^{k}}{\partial s^k}\frac{1}{\zeta(\sigma-\frac{1}{2}+s)}. \end{aligned}$$
The residue at the pole $s=1-\sigma+i\beta,$ is 
$$\sum_{\beta: \zeta(\rho)=0}\left(\frac{\sigma-i\beta}{1-\sigma+i\beta}\right)^n\frac{1}{\zeta'(\rho)(1-\sigma+i\beta)(\sigma-i\beta)},$$
and at the pole $s=\frac{1}{2}-\sigma-2k$ is
$$\sum_{k\ge1}\left(\frac{\frac{1}{2}+\sigma+2k}{\frac{1}{2}-\sigma-2k}\right)^n\frac{1}{\zeta'(-2k)(\frac{1}{2}-\sigma-2k)(\frac{1}{2}+\sigma+2k)},$$
The residue at the pole $n=0,$ $s=0,$ is $-1/\zeta(\sigma-\frac{1}{2}).$
\end{proof}
\begin{proof}[Proof of Corollary~\ref{corollary:Cor1}] This result readily follows from application of Theorem 1.1 to Lemma 2.1 with $X=\mathbb{R}.$ In the first part of the theorem, note from [14, pg.191, Theorem 8.7], if $\sigma>1,$
$$\left|\frac{1}{\zeta(s)}\right|\le\frac{\zeta(\sigma)}{\zeta(2\sigma)}.$$ Hence
$$\frac{1}{\left|\zeta(s)\right|^2(t^2+\frac{1}{4})}=O\left(\frac{1}{|t|^2}\right),$$ as $t\rightarrow\infty,$ and $1/\zeta(\sigma+it)\in L^2(\mu),$ for $\sigma>1.$ The convergence of the series $\sum_{k}|\bar{a}_n|^2$ follows immediately from [10, pg.580, Lemma 12.6]. In the second part of the theorem, note from [14, pg.377] or [14, pg.372]
$$\frac{1}{\zeta(s)}=O\left(\frac{|s|}{\sigma-\frac{1}{2}}\right).$$ Hence
$$\frac{1}{\left|\zeta(s)\right|^2(t^2+\frac{1}{4})}=O\left(\frac{|s|^2}{|t|^2}\right)=O(1),$$ as $t\rightarrow\infty,$ and $1/\zeta(\sigma+it)\notin L^2(\mu),$ for $\frac{1}{2}<\sigma<1.$\end{proof}

\begin{proof}[Proof of Theorem~\ref{thm:thm2}] It is clear that
$$\cos^{v}(2\tan^{-1}(2y))=\left(\frac{1-4y^2}{1+4y^2}\right)^v=O(1).$$ Comparing with [7, Theorem 1.2] we see our function belongs to $L^2(\mu).$
We compute that
$$ \begin{aligned}
&\tilde{a}_n=\frac{1}{2\pi}\int_{-\pi}^{\pi}f(\frac{1}{2}\tan(\frac{\phi}{2}))e^{in\phi}d\phi \\
&=\frac{1}{2\pi}\int_{-\pi}^{\pi}\zeta(\sigma+\frac{i}{2}\tan(\frac{\phi}{2}))\cos^v(\frac{\phi}{2}) e^{in\phi}d\phi \\
&=\frac{1}{2\pi}\left(\int_{0}^{\pi}\zeta(\sigma+\frac{i}{2}\tan(\frac{\phi}{2})) \cos^v(\frac{\phi}{2}) e^{in\phi}d\phi +\int_{-\pi}^{0}\zeta(\sigma+\frac{i}{2}\tan(\frac{\phi}{2}))\cos^v(\frac{\phi}{2})e^{in\phi}d\phi\right)\\
&=\frac{1}{2\pi}\left(\int_{0}^{\pi}\zeta(\sigma+\frac{i}{2}\tan(\frac{\phi}{2})) \cos^v(\frac{\phi}{2}) e^{in\phi}d\phi +\int_{0}^{\pi}\zeta(\sigma-\frac{i}{2}\tan(\frac{\phi}{2}))\cos^v(\frac{\phi}{2})e^{-in\phi}d\phi\right)\\
&=\frac{1}{2\pi}\left(\int_{0}^{\pi}\zeta(\sigma+\frac{i}{2}\tan(\frac{\phi}{2})) \cos^v(\frac{\phi}{2}) e^{in\phi}d\phi +\int_{0}^{\pi}\zeta(\sigma-\frac{i}{2}\tan(\frac{\phi}{2}))\cos^v(\frac{\phi}{2})e^{-in\phi}d\phi\right)\\
&=\frac{1}{\pi}\int_{0}^{\pi} \cos^v(\frac{\phi}{2}) \sum_{k\ge1}k^{-\sigma}\cos\left(\frac{1}{2}\tan(\frac{\phi}{2})\log(k)-n\phi\right)d\phi\\
&=\frac{1}{\pi}\int_{0}^{\pi} \cos^v(\frac{\phi}{2}) \cos\left(n\phi\right) d\phi +\frac{1}{\pi}\int_{0}^{\pi} \cos^v(\frac{\phi}{2}) \sum_{k>1}k^{-\sigma}\cos\left(\frac{1}{2}\tan(\frac{\phi}{2})\log(k)-n\phi\right) d\phi \\
&=\frac{2}{\pi}\int_{0}^{\pi/2} \cos^v(\phi) \cos\left(n2\phi\right) d\phi +\frac{2}{\pi}\int_{0}^{\pi/2} \cos^v(\phi) \sum_{k>1}k^{-\sigma}\cos\left(\frac{1}{2}\tan(\phi)\log(k)-n2\phi\right) d\phi. \\
\end{aligned}$$
Now by [8, pg.397] for $\Re(v)>0,$ we have
\begin{equation}\int_{0}^{\pi/2}\cos^{v-1}(y)\cos(by)dy=\frac{\pi\Gamma(v)}{\Gamma(\frac{v+b+1}{2})\Gamma(\frac{v-b+1}{2})}.\end{equation}
Let $\mathbb{Z}^{-}$ denote the set of negative integers. Then, by [8, pg.423] with $a>0,$ $\Re(v)>-1,$ $\frac{v+\gamma}{2}\neq\mathbb{Z}^{-},$
\begin{equation}\int_{0}^{\pi/2}\cos^v(y)\cos(a\tan(y)-\gamma y)dy=\frac{\pi a^{v/2}}{2^{v/2+1}}\frac{W_{\gamma/2,-\frac{v+1}{2}}(2a)}{\Gamma(1+\frac{v+\gamma}{2})}.\end{equation}
Hence, if we put $b=2n$ and replace $v$ by $v+1$ in (2.4), and select $a=\frac{1}{2}\log(k)$ and $\gamma=2n$ in (2.5), we find 
\begin{equation}\tilde{a}_n=\frac{2\Gamma(v+1)}{\Gamma(\frac{v}{2}+n+1)\Gamma(\frac{v}{2}-n+1)}+\frac{\pi }{2^{v/2+1}}\sum_{k>1} k^{-\sigma}\left(\frac{\log(k)}{2}\right)^{v/2}\frac{W_{n,-\frac{v+1}{2}}(\log(k))}{\Gamma(1+\frac{v}{2}+n)}.\end{equation} Hence $v$ cannot be a negative even integer.

The interchange of the series and integral is justified by absolute convergence for $\sigma>\frac{1}{2}.$ To see this, note that [8, pg.1026, eq.(9.227), eq.(9.229)]
$$W_{\gamma,\mu}(z)\sim e^{-z/2}z^{\gamma},$$ as $|z|\rightarrow\infty,$ and
$$W_{\gamma,\mu}(z)\sim (\frac{4z}{\gamma})^{1/4}e^{-\gamma+\gamma\log(\gamma)}\sin(2\sqrt{\gamma z}-\gamma\pi-\frac{\pi}{4}),$$ as $|\gamma|\rightarrow\infty.$
Using (2.6) as coefficients for $n<0$ is inadmissible, due to the resulting sum over $n$ being divergent. On the other hand, it can be seen that
$$\begin{aligned} &\tilde{a}_n=\frac{1}{2\pi}\int_{\mathbb{R}}e^{2in\tan^{-1}(2y)}\frac{\zeta(\sigma+iy)\cos^{v}(\tan^{-1}(2y))dy}{(\frac{1}{4}+y^2)}\\
&=\frac{1}{2\pi i}\int_{(\frac{1}{2})}\frac{\zeta(\sigma-\frac{1}{2}+s)2(2s^2-2s+1)}{(2s(1-s))^2}\left(\frac{s}{1-s}\right)^nds\\
&=\frac{1}{2\pi i}\int_{(\frac{1}{2})}\frac{\zeta(\sigma+\frac{1}{2}-s)2(2s^2-2s+1)}{(2s(1-s))^2}\left(\frac{1-s}{s}\right)^nds  . \end{aligned}$$
We will only use the residues at the pole $s=0$ when $n<0$ and $s=\sigma-\frac{1}{2},$ and outline the details to obtain an alternative expression for the $\tilde{a}_n$ for $n\ge0.$ The integrand has a simple pole at $s=\sigma-\frac{1}{2},$ a pole of order $n+2$ at $s=0,$ and when $n<0$ there is a simple pole when $n=-1,$ at $s=0.$ The residue at the pole $s=0$ for $n\ge0$ is computed as
\begin{equation}\begin{aligned}
&\frac{1}{n!}\lim_{s\rightarrow0}\frac{d^{n+1}}{ds^{n+1}}s^{n+2}\left(\frac{\zeta(\sigma+\frac{1}{2}-s)2(2s^2-2s+1)}{(2s(1-s))^2}\left(\frac{1-s}{s}\right)^n\right) \\
&=\frac{1}{n!2}\lim_{s\rightarrow0}\frac{d^{n+1}}{ds^{n+1}}\left(\zeta(\sigma+\frac{1}{2}-s)(2s^2-2s+1)(1-s)^{n-2}\right). \end{aligned}\end{equation}
And because the resulting sum is a bit cumbersome, we omit this form in our stated theorem. The residue at the simple pole when $n=-1,$ at $s=0$ is $\frac{1}{2}\zeta(\sigma+\frac{1}{2}).$ Collecting our observations tells us that if $n<0,$
$$\tilde{a}_n=\frac{(2\sigma^2-4\sigma+\frac{5}{2})}{2(\sigma-\frac{1}{2})^2(\frac{3}{2}-\sigma)^2}\left(\frac{\frac{3}{2}-\sigma}{\sigma-\frac{1}{2}}\right)^n.$$

\end{proof}

\section{Riemann xi function}
The Riemann xi function is given by $\xi(s):=\frac{1}{2}s(s-1)\pi^{-\frac{s}{2}}\Gamma(\frac{s}{2})\zeta(s),$ and $\Xi(y)=\xi(\frac{1}{2}+iy).$ In many recent works [4, 5], Riemann xi function integrals have been shown to have interesting evaluations. (See also [11] for an interesting expansion for the Riemann xi function.) The classical application is in the proof of Hardy's theorem that there are infinitely many non-trivial zeros on the line $\Re(s)=\frac{1}{2}.$ \par We will need to utilize Mellin transforms to prove our theorems. By Parseval's formula [12, pg.83, eq.(3.1.11)], we have
\begin{equation}\int_{0}^{\infty}f(y)g(y)dy=\frac{1}{2\pi i}\int_{(r)}\mathfrak{M}(f(y))(s)\mathfrak{M}(g(y))(1-s)ds,\end{equation}
provided that $r$ is chosen so that the integrand is analytic, and 
$$\int_{0}^{\infty}y^{s-1}f(y)dy=:\mathfrak{M}(f(y))(s).$$
From [12, pg.95, eq.(3.3.27)] with $n\ge0,$ $x>1,$ $c>0,$ we have
\begin{equation}\frac{1}{2\pi i}\int_{(c)}\frac{x^s}{s^{n+1}}ds=\frac{(\log(x))^n}{n!}.\end{equation} Now it is known [6, pg.207--208] that for any $\Re(s)=u\in \mathbb{R},$ 
\begin{equation}\Theta(y)=\frac{1}{2\pi i}\int_{(u)}\xi(s)y^{-s}ds,\end{equation}
where 
\begin{equation}\Theta(y):=2y^2\sum_{n\ge1}(2\pi^2 n^4y^2-3\pi n^2)e^{-\pi n^2 y^2},\end{equation} 
for $y>0.$ Define the operator $\mathfrak{D}_{n, y}(f(y)):=\underbrace{y\frac{\partial }{\partial y}\dots y\frac{\partial }{\partial y}}_{n}(f(y)).$ 
\begin{theorem} For real numbers $x\in\mathbb{R},$
$$\Xi(x)=\frac{1}{(\frac{1}{4}+x^2)}\sum_{n\in\mathbb{Z}}\ddot{a}_ne^{-2in\tan^{-1}(2x)},$$
where $\ddot{a}_0=0,$ and for $n\ge1,$ 
$$\ddot{a}_n=\frac{(-1)^{n}}{(n-1)!}\int_{0}^{1}\log^{n-1}(y)\mathfrak{D}_{n, y}(\Theta(y))dy,$$ and
$$\ddot{a}_{-n}=-\frac{(-1)^n}{(n-1)!}\sum_{n-1\ge k \ge0}\binom{n-1}{k}\frac{n!}{(k+1)!}\xi^{(k)}(0).$$
\end{theorem}

\begin{proof}
Applying the operator $\mathfrak{D}_{n, y}$ to (3.3)--(3.4), then applying the resulting Mellin transform with (3.2) to (3.1), we have for $c<1,$ $n\ge1,$
\begin{equation}\frac{(-1)^{n}}{(n-1)!}\int_{0}^{1}\log^{n-1}(y)\mathfrak{D}_{n, y}(\Theta(y))dy=\frac{1}{2\pi i}\int_{(c)}\left(\frac{s}{1-s}\right)^n\xi(s)ds.\end{equation}
On the other hand,
\begin{equation}\begin{aligned} &\ddot{a}_n=\frac{1}{2\pi}\int_{\mathbb{R}}e^{2in\tan^{-1}(2y)}\frac{(\frac{1}{4}+y^2)\Xi(y)dy}{(\frac{1}{4}+y^2)}dy=\frac{1}{2\pi i}\int_{(\frac{1}{2})}\left(\frac{s}{1-s}\right)^n\xi(s)ds \\
&= \frac{1}{2\pi i}\int_{(\frac{1}{2})}\left(\frac{s}{1-s}\right)^n\pi^{-s/2}\frac{s}{2}(s-1)\zeta(s)\Gamma(\frac{s}{2})ds .\end{aligned}\end{equation}
This gives the coefficients for $n\ge1.$ If we place $n$ by $-n$ in the integrand of (3.6), we see that there is a pole of order $n,$ $n>0,$ at $s=0.$ These residues are computed in the same way as before, and so we leave the details to the reader. Hence, for $n>0,$ $-2<r'<0,$
\begin{equation}\begin{aligned}&\ddot{a}_{-n}=\frac{1}{2\pi i}\int_{(\frac{1}{2})}\left(\frac{1-s}{s}\right)^n\xi(s)ds\\
&=\frac{(-1)^n}{(n-1)!}\sum_{n-1\ge k \ge0}\binom{n-1}{k}\frac{n!}{(k+1)!}\xi^{(k)}(0)+\frac{1}{2\pi i}\int_{(r')}\left(\frac{1-s}{s}\right)^n\xi(s)ds\\
&=\frac{(-1)^n}{(n-1)!}\sum_{n-1\ge k \ge0}\binom{n-1}{k}\frac{n!}{(k+1)!}\xi^{(k)}(0) .\end{aligned} \end{equation}
In the third line we implemented the fact that the remaining residue from the poles of $\Gamma(\frac{s}{2})$ at negative even integers is zero due to the trivial zeros of $\zeta(s).$\end{proof}
Now according to Coffey [1, pg.527], $\xi^{(n)}(0)=(-1)^n\xi^{(n)}(1),$ which may be used to recast Theorem 3.1 in a slightly different form. The integral formulae obtained in [2, pg.1152, eq.(28)] (and another form in [9, pg.11106, eq.(12)]) bear some resemblance to the integral contained in (3.5). It would be interesting to obtain a relationship to the coefficients $\ddot{a}_n.$ Next we give a series evaluation for a Riemann xi function integral.

\begin{corollary} If the coefficients $\ddot{a}_n$ are as defined in Theorem 3.1., then
$$\int_{\mathbb{R}}(\frac{1}{4}+y^2)^2\Xi^2(y)d\mu=\sum_{n\in\mathbb{Z}}|\ddot{a}_n|^2.$$
\end{corollary}

\begin{proof} This is an application of Theorem 3.1 to Lemma 2.1 with $X=\mathbb{R}.$ \end{proof}

\section{On the partial Fourier series}

Here we make note of some interesting consequences of our computations related to the partial sums of our Fourier series. First, we recall [10, pg.69] that 

\begin{equation}\sum_{n=-N}^{N}a_ne^{inx}=\frac{1}{2\pi}\int_{-\pi}^{\pi}f(x-y)D_{N}(y)dy,\end{equation}
where $$D_{N}(x)=\frac{\sin((N+\frac{1}{2})x)}{\sin(\frac{x}{2})}.$$ Now making the change of variable $y=2\tan^{-1}(2y),$ we find (4.1) is equal to
$$\frac{1}{2\pi}\int_{\mathbb{R}}f(x-2\tan^{-1}(2y))\frac{D_{N}(2\tan^{-1}(2y))}{\frac{1}{4}+y^2}dy.$$
Recall [10, pg.71] that $K_{N}(x)$ is the Fej\'er kernel if
$$K_{N}(x)=\frac{1}{N+1}\sum_{n=0}^{N}D_{n}(x).$$

\begin{theorem} Let $K_{N}(x)$ denote the Fej\'er kernel. Then, assuming the Riemann hypothesis,
$$\lim_{N\rightarrow\infty}\frac{1}{2\pi}\int_{\mathbb{R}}\frac{K_{N}(x_0-2\tan^{-1}(2y))}{\zeta(\sigma+iy)(\frac{1}{4}+y^2)}dy=\frac{1}{\zeta(\sigma+\frac{i}{2}\tan(\frac{x_0}{2}))},$$
for $x_0\in(-\pi,\pi),$ $\frac{1}{2}<\sigma<1.$
\end{theorem}
\begin{proof} Notice that $1/\zeta(\sigma+\frac{i}{2}\tan(\frac{y}{2}))$ is continuous for $y\in(-\pi,\pi)$ if there are no singularities for $\frac{1}{2}<\sigma<1.$ Hence, we may apply [10, pg.29, Theorem 1.26] to find $1/\zeta(\sigma+\frac{i}{2}\tan(\frac{y}{2}))$ would then be Riemann integrable on $(-\pi,\pi)$ if there are no singularities for $\frac{1}{2}<\sigma<1.$ It is also periodic in $\pi.$ Applying Fej\'er's theorem [10, pg.73, Theorem 1.59] with $f(y)=1/\zeta(\sigma+\frac{i}{2}\tan(\frac{y}{2}))$ implies the result.
\end{proof}
Note that if $1/\zeta(\sigma+\frac{i}{2}\tan(\frac{y}{2}))$ has even finitely many points of discountinuity for $\frac{1}{2}<\sigma<1,$ we would not be able to apply Fej\'er's theorem. This is because the function is unbounded by Montgomery's omega result [14, pg.209], and therefore not Riemann integrable by [10, pg.31, Proposition 1.29].
\section{Concluding remarks}
The Fourier series for the Riemann zeta function contained herein, just like those in [7], are pointwise convergent. Seeing as how there exists a Fourier series for $\zeta(\sigma+it)$ in the region $\frac{1}{2}<\sigma<1,$ that is pointwise convergent, it would be interesting if one existed that were absolutely convergent. Wiener's result [15, pg.14, Lemma IIe] says the following:

\begin{lemma}(Wiener [15]) Suppose $f(x)$ has an absolutely convergent Fourier series and $f(x)\neq0$ for all $x\in \mathbb{R}.$ Then its reciprocal $1/f(x)$ also has an absolutely convergent Fourier series.
\end{lemma}
Therefore, an application of the Riemann hypothesis would then imply the existence of an absolutely convergent Fourier series for $1/\zeta(\sigma+it),$ when $\frac{1}{2}<\sigma<1.$

1390 Bumps River Rd. \\*
Centerville, MA
02632 \\*
USA \\*
E-mail: alexpatk@hotmail.com, alexepatkowski@gmail.com

\end{document}